\documentclass{amsart}
\usepackage[english]{babel}
\usepackage[latin1]{inputenc}
\usepackage[T1]{fontenc}
\usepackage[]{mathrsfs}
\usepackage[]{amsmath}
\usepackage[]{amssymb}
\usepackage[]{amsthm}
\usepackage[]{graphicx}

\usepackage{natbib}
\bibpunct[: ]{[}{]}{;}{n}{}{:}

\author{Atte Aalto} \title[Spatial discretization in {K}alman filtering]{Spatial
  discretization error in {K}alman filtering for discrete-time infinite
  dimensional systems} 

\newcommand{\norm}[1]{\left| \! \left| #1 \right| \! \right|}
\newcommand{\ip}[1]{\left<#1\right>}

\newcommand{\Ex}[1]{\mathbb{E}\! \left(#1\right)}
\newcommand{\Cov}[1]{\textup{Cov}\left[#1\right]}

\newcommand{\tr}{\textup{tr}}
\newcommand{\bbm}[1]{\left[\begin{matrix} #1 \end{matrix}\right]}
\newcommand{\sbm}[1]{\left[\begin{smallmatrix} #1
             \end{smallmatrix}\right]}
\newcommand{\sX}{\mathcal{X}}
\newcommand{\sY}{\mathcal{Y}}

\begin{document}
\maketitle


\begin{abstract}
  \noindent We derive a reduced-order state estimator for
  discrete-time infinite dimensional linear systems with finite
  dimensional Gaussian input and output noise. This state estimator is
  the optimal one-step estimate that takes values in a fixed finite
  dimensional subspace of the system's state space --- consider, for
  example, a Finite Element space. We then derive a Riccati difference
  equation for the error covariance and use sensitivity analysis to
  obtain a bound for the error of the state estimate due to the state
  space discretization.
\end{abstract}

\newtheorem{definition}{Definition}[section]
\newtheorem{proposition}[definition]{Proposition}
\newtheorem{theorem}[definition]{Theorem}
\newtheorem{remark}[definition]{\it Remark}
\newtheorem{example}[definition]{Example}
\newtheorem{lemma}[definition]{Lemma}
\newtheorem{corollary}[definition]{Corollary}
\newtheorem{algorithm}[definition]{\it Algorithm}

\section{Introduction}

In this paper, we consider the state estimation problem for infinite
dimensional discrete time linear systems with finite dimensional
Gaussian input and output noise. The objective is to find the optimal
one-step state estimate from a given subspace of the original state
space (for example a Finite Element space). We shall also find a bound
for the error due to the spatial discretization to the state estimate
at the infinite time limit.

The dynamics of the system under consideration is given by
\begin{equation} \label{eq:system}
\begin{cases} 
x_k=Ax_{k-1}+Bu_k, \\
y_k=Cx_k+w_k, \\
x_0\sim N(m,S_0)
\end{cases}
\end{equation}
where $x_k \in \sX$, $A \in \mathcal{L}(\sX)$, $B \in
\mathcal{L}(\mathbb{C}^q,\sX)$, and $C \in
\mathcal{L}(\sX,\mathbb{C}^m)$. The state space $\sX$ is a separable
Hilbert space.
The noise processes are assumed to be Gaussian, $u_k \sim N(0,U)$ and
$w_k \sim N(0,R)$ where $U \in \mathbb{R}^{q \times q}$ and $R \in
\mathbb{R}^{m \times m}$ are positive-definite and symmetric. It is
also assumed that $u$, $w$, and $x_0$ are mutually independent, and the
noises at different times are independent.

When measurements $y_j$ for $j=1,...,k$ are known, the state estimate
$\hat{x}_k$ minimizing the conditional expectation
$\Ex{\norm{\hat{x}_k-x_k}_{\sX}^2 \Big|\{y_j, \, j=1,...,k\} }$ is
given by $\hat x_k=\Ex{x_k|\{y_j, \, j \le k\}}$. In the presented
Gaussian case, the conditional expectation $\hat x_k$ can be computed
recursively from $\hat x_{k-1}$ and $y_k$. This recursive scheme is
known as the Kalman filter, originally presented in \cite{Kalman} in
the finite dimensional setting. For infinite dimensional systems, the
generalization is straightforward and it can be done, for example,
using the presentation by Bogachev \cite[Section~3.10]{Bogachev} or
the more explicit presentation \cite{Krug} by Krug. Let us present a
short introduction. It is well known that linear combinations of
Gaussian random variables
are also Gaussian random variables. Further, if $\bbm{h_1 \\
  h_2} \sim N\left( \bbm{m_1 \\ m_2},\bbm{P_{11} & P_{12} \\ P_{12}^*
    & P_{22}} \right)$ where $h_1 \in \sX$ and $h_2$ is finite
dimensional, then
\begin{equation} \label{eq:cond}
\Ex{h_1|h_2} =m_1+P_{12}P_{22}^+(h_2-m_2)
\end{equation}
and
\begin{equation} \label{eq:cond_cov}
\Cov{h_1-\Ex{h_1|h_2},h_1-\Ex{h_1|h_2}}=P_{11}-P_{12}P_{22}^+P_{12}^*.
\end{equation}
Remark that
$\Cov{\Ex{h_1|h_2},\Ex{h_1|h_2}}=P_{12}P_{22}^+P_{12}^*$
so that in fact,
\begin{align} \label{eq:pythagoras}
&\Cov{h_1 -\Ex{h_1|h_2},h_1 -\Ex{h_1|h_2}} \\ & \hspace{10mm} \nonumber =\Cov{h_1,h_1}-\Cov{\Ex{h_1|h_2},\Ex{h_1|h_2}}.
\end{align}

Applying \eqref{eq:cond} and \eqref{eq:cond_cov} to the jointly
Gaussian random variable $[x_k,y_1,...,y_k]$ and the block matrix
inversion formula
\begin{equation} \label{eq:inv}
\! \sbm{F \! & G \! \\ G^T & H}^{\!-1} \!\! =\sbm{F^{-1} \!
  +F^{-1}G(H \!-G^TF^{-1}G)^{-1}G^TF^{-1}  & -F^{-1}G(H
  \!-G^TF^{-1}G)^{-1} \\ -(H-G^TF^{-1}G)^{-1}G^TF^{-1} &
  (H-G^TF^{-1}G)^{-1}}
\end{equation}
to $P_{22} \, \widehat = \, \Cov{[y_1,...,y_k],[y_1,...,y_k]}$
eventually leads to the \emph{full state Kalman filter} equations
\begin{equation} \label{eq:KF}
\hat x_k=A\hat x_{k-1}+K_k^{(F)}(y_k-CA\hat x_{k-1})
\end{equation}
where $K_k^{(F)}$ for $k=1,2,...$ are called \emph{Kalman gains}, and
they are given by $ K_k^{(F)}=\tilde P_k^{(F)}C^*(C\tilde
P_k^{(F)}C^*+R)^{-1}$, and the \emph{Riccati difference equation} (RDE)
\begin{equation} \label{eq:full_RDE}
\begin{cases}
  \tilde P_k^{(F)}=AP_{k-1}^{(F)}A^*+BUB^*, \\
  P_k^{(F)}=\tilde P_k^{(F)}-\tilde P_k^{(F)}C^*(C\tilde P_k^{(F)}C^*+R)^{-1}C\tilde P_k^{(F)}.
\end{cases}
\end{equation}
Here $P_k^{(F)}=\Cov{x_k-\hat x_k,x_k-\hat x_k}$ is the
\emph{(estimation) error covariance} and $\tilde
P_k^{(F)}=\Cov{x_k-\Ex{x_k|[y_1,...,y_{k-1}]},x_k-\Ex{x_k|[y_1,...,y_{k-1}]}}$
is the \emph{prediction error covariance}. The initial values are
$\hat x_0=m$ and $P_0^{(F)}=S_0$.
The superscript $(F)$ refers to full Kalman filter estimate and it is
used for later purposes.


Numerical implementation of the Kalman filter to infinite dimensional
systems requires discretization of the state space. If the
implementation is then carried out directly to the discretized system,
the result is not optimal. In particular, if the state estimation is
performed online, the restrictions in computing power might prevent
using a very fine mesh for the simulations. In such cases it is
beneficial to take the discretization error into account in the state
estimation. The purpose of this paper is to derive the optimal
one-step state estimate that takes values in the discretized state
space, and to analyze the discrepancy between the proposed state
estimate and the full state Kalman filter estimate.


We tackle this task in Section~\ref{sec:estimate} by first fixing the
structure of the filter in \eqref{eq:structure}. In the spirit of
Kalman filtering, we require that the $k^{\textrm{th}}$ estimate
depends only on the previous estimate and the current measured output
$y_k$.  We then find the expression for a filter with such structure.
The rest of the paper is organized as follows: In
Section~\ref{sec:asymptotics}, we derive a Riccati difference equation
for the estimation error covariance for the proposed method.  Compared
to \eqref{eq:full_RDE}, this equation contains an additional term due
to the discretization.  In Section~\ref{sec:error}, we use sensitivity
analysis for algebraic Riccati equations --- developed by Sun
in~\cite{Sun} --- to determine a bound for the error due to the
discretization at the infinite time limit. In short, it is shown that
when the approximation properties of the subspace improve at some rate
as the spatial discretization is refined, then the finite dimensional
state estimate converges to the full state Kalman filter estimate at
least with the same convergence rate.  In Section \ref{sec:example},
the proposed method is implemented to one dimensional wave equation
with damping, and the result is compared with the Kalman filter that
does not take into account the spatial discretization error. 


The ``engineer's approach'', \emph{i.e.}, the direct Kalman filter
implementation to the discretized system is studied in
\cite{Bensoussan} by Bensoussan and in \cite{Galerkin_filter} by
Germani \emph{et al}. The latter contains a convergence result for the
finite dimensional state estimate (in continuous time) with a
convergence rate estimate. They also show convergence of the solutions
of the corresponding Riccati differential equations in the space of
continuous Hilbert-Schmidt operator-valued functions.  A method where
the discretization error is taken into account is proposed by
Pik\-ka\-rai\-nen in \cite{Pikkarainen}. Their approach is based on
keeping track of the discretization error mean and covariance. Then
with certain approximations on the error distributions, they too end
up with a one-step method that is numerically implemented in
\cite{Hutt_Pikk} by Hut\-tu\-nen and Pik\-ka\-rai\-nen.

Our method is very closely related to the reduced-order filtering
methods that have been studied since the introduction of the Kalman
filter itself; see \emph{e.g.}, \cite{Bernstein, Bernstein_DPS,
  simon,sims,stub}.  The articles by Bernstein and Hyland,
\cite{Bernstein,Bernstein_DPS} yield a state estimator similar to ours
for continuous time. They obtain algebraic optimality equations for
the error covariance and Kalman gain limits as the time index $k \to
\infty$, in terms of ``optimal projections''. Our solution is somewhat
more straightforward, and we obtain the error covariances and Kalman
gains for all time steps. A similar method is developed by Simon in
\cite{simon} with a more restrictive assumption on the filter
structure. For a more thorough introduction and review on the earliest
results on reduced-order filtering techniques, we refer to \cite{stub}
by Stubberud and Wismer and to \cite{sims} by Sims.

Infinite dimensional Kalman filter has numerous applications. The
practical application that motivated the paper \cite{Pikkarainen} is
the electrical impedance process tomography, studied by Sepp\"anen
\emph{et al}. in \cite{EIT}. Infinite dimensional Kalman filter
  implementation to optical tomography problem can be found in
  \cite{opt_tom} by Hiltunen \emph{et al}.  Quasiperiodic phenomena is
  studied by Solin and S\"arkk\"a in \cite{Solin} using the infinite
  dimensional Kalman filter. They use a weather prediction model and
  fMRI brain imaging as example cases. The numerical treatment is done
  using truncated eigenbasis approach instead of using FEM as in the
  example of this article.





\subsubsection*{Notation}  

We denote by $\mathcal{L}(\sX_1,\sX_2)$ the space of bounded linear
operators from $\sX_1$ to $\sX_2$, and
$\mathcal{L}(\sX)=\mathcal{L}(\sX,\sX)$. The subspace of self-adjoint
operators in $\sX$ is denoted by $\mathcal{L}^*(\sX)$. The spectrum of
an operator is denoted by $\sigma(\cdot)$. The sigma algebra generated
by a random variable (or random variables) is denoted by
$\mathcal{S}(\cdot)$. The Moore-Penrose pseudoinverse of a matrix $T$
is denoted by $T^+$.

The covariance of square integrable random variables $x_1 \in \sX_1$
and $x_2 \in \sX_2$ is the operator in $\mathcal{L}(\sX_2,\sX_1)$
defined for $h \in \sX_2$ by \newline
$\Cov{x_1,x_2}h := \Ex{(x_1-\Ex{x_1})\ip{x_2-\Ex{x_2},h}_{\sX_2}}$.


\section{The reduced-order state estimate} \label{sec:estimate}

Let $\Pi_s:\sX \to \sX$ be an orthogonal projection from the state
space $\sX$ (a separable, complex Hilbert space) to an $n$-dimensional
subspace of $\sX$ (\emph{e.g.}, a finite element space). Assume we have a
coordinate system in $\mathbb{C}^n$ associated to this subspace, such
that the inner product is preserved, and denote by $\Pi:\sX \to
\mathbb{C}^n$ the representation of the projection $\Pi_s$ in this
coordinate system. That is, $\ip{\Pi_s x_1,\Pi_s x_2}_{\sX}=\ip{\Pi
  x_1,\Pi x_2}_{\mathbb{C}^n}$ for $x_1,x_2 \in \sX$. Then it holds that
$\Pi \Pi^*=I \in \mathbb{C}^{n \times n}$ and $\Pi^*\Pi=\Pi_s$.

Finding an exact solution to the estimation problem of the finite
dimensional $\Pi x_k$ would require solving the full state Kalman
filtering problem and then projecting the estimate by $\Pi$. This, of
course, doesn't make much practical sense.
As mentioned above, we want to find the optimal state estimate $\tilde
x_k$ in $\Pi_s \sX$ that can be computed from the previous state
estimate $\tilde x_{k-1}$ and the current measurement $y_k$. More
precisely, we want to obtain $\tilde x_k$'s satisfying
\begin{equation} \label{eq:structure}
\begin{cases}
\tilde x_0 = \Pi m, \\
\tilde x_k = \Pi\Ex{ x_k | \tilde x_{k-1}, y_k}, \qquad k \ge 1,
\end{cases}
\end{equation}
where $x_k$ satisfy \eqref{eq:system}.  One thing to notice here is
that in contrast to the full state filtering, the conditioning is not
done over a filtration, because --- loosely speaking --- we lose some
information when we only take into account the last measurement and
the last estimate of the state projection.  Without loss of
generality, we may assume that $m=0$ (see Remark~\ref{rem:m_zero}).
Note that this
also implies $\Ex{x_k}=0$ and further, $\Ex{\tilde x_k}=0$ and
$\Ex{y_k}=0$ for all $k \ge 1$.

We then proceed to find a concrete representation for $\tilde
x_k$. From \eqref{eq:structure} it can be inductively deduced that
$[x_{k-1},\tilde x_{k-1}]$ is Gaussian and from \eqref{eq:system},
also $[x_k,\tilde x_{k-1},y_k]$ is Gaussian.  The reasoning leading to
the full state Kalman filter equations utilizing equations
\eqref{eq:cond} and \eqref{eq:cond_cov} together with the block matrix
inversion formula \eqref{eq:inv} can be generalized for any Gaussian
random variable $[h_1,h_2,h_3]$ with $h_1 \in \sX$, and $h_2$ and
$h_3$ finite dimensional, to obtain
\begin{align} \label{eq:cond3}
  \Ex{h_1|h_2,h_3}=&\Ex{h_1|h_2}+\Cov{h_1-\Ex{h_1|h_2},h_3-\Ex{h_3|h_2}}
  \times \\ \nonumber
  &\times \Cov{h_3-\Ex{h_3|h_2},h_3-\Ex{h_3|h_2}}^{-1}(h_3-\Ex{h_3|h_2}).
\end{align}
The corresponding equation can be obtained for the covariance
operator. The full state Kalman filter equations \eqref{eq:KF} and
\eqref{eq:full_RDE} are obtained by applying \eqref{eq:cond3} to
$h_1=x_k$, $h_2=[y_1,...,y_{k-1}]$, and $h_3=y_k$. In what follows, we
obtain $\tilde x_k$ by applying \eqref{eq:cond3} to $h_1=x_k$,
$h_2=\tilde x_{k-1}$, and $h_3=y_k$.

Since $m=0$, there exists an operator $Q_{k-1} \in
\mathcal{L}(\mathbb{R}^n,\sX)$ such that
\begin{equation} \label{eq:Q_def} \Ex{x_{k-1}|\tilde
    x_{k-1}}=Q_{k-1}\tilde x_{k-1}
\end{equation}
and the \emph{(estimation) error covariance} 
\begin{equation} \label{eq:err_cov}
P_{k-1}:=\Cov{x_{k-1}-Q_{k-1}\tilde x_{k-1},x_{k-1}-Q_{k-1}\tilde x_{k-1}}.
\end{equation}
Using these we can make an orthogonal decomposition of the state
\begin{equation} \nonumber x_{k-1}=\Ex{x_{k-1}|\tilde x_{k-1}}+
  \left(x_{k-1}-\Ex{x_{k-1}|\tilde x_{k-1}} \right)=:Q_{k-1}\tilde
  x_{k-1}+v_{k-1}
\end{equation}
where $v_{k-1} \sim N\big(0,P_{k-1}\big)$ and it is independent of the
estimate $\tilde x_{k-1}$. Together with \eqref{eq:system}, this gives
decompositions for the state $x_k$ and output $y_k$:
\begin{equation} \label{eq:decomposition}
\begin{cases}
x_k=Ax_{k-1}+Bu_k=A(Q_{k-1}\tilde x_{k-1}+v_{k-1})+Bu_k, \\
y_k=Cx_k+w_k=C(A(Q_{k-1}\tilde x_{k-1}+v_{k-1})+Bu_k)+w_k
\end{cases}
\end{equation}
from which one can deduce $\Ex{ x_k | \tilde x_{k-1}}= A Q_{k-1}\tilde
x_{k-1}$ and $\Ex{ y_k | \tilde x_{k-1}}= CA Q_{k-1}\tilde
x_{k-1}$.

Then we need the two covariances in \eqref{eq:cond3}. To this end,
define the \emph{prediction error covariance} for which we get a
representation from \eqref{eq:decomposition},
\begin{equation} \label{eq:pred_cov} \tilde P_k :=
  \Cov{x_k-\Ex{x_k|\tilde x_{k-1}},x_k-\Ex{x_k|\tilde x_{k-1}}}
=AP_{k-1}A^*+BUB^*.
\end{equation}
Using the two equations in \eqref{eq:decomposition}, we get
\begin{equation} \nonumber \Cov{x_k-\Ex{x_k|\tilde
      x_{k-1}},y_k-\Ex{y_k|\tilde
      x_{k-1}}}=\tilde P_kC^*
\end{equation}
and the covariance of output prediction error from the second equation in
\eqref{eq:decomposition}
\begin{equation} \nonumber \Cov{y_k-\Ex{y_k|\tilde
      x_{k-1}},y_k-\Ex{y_k|\tilde x_{k-1}}}=C \tilde P_k C^*+R.
\end{equation}
Now we have all the components for obtaining $\tilde x_k$ by
\eqref{eq:cond3},
\begin{equation} 
\label{eq:kf_proj}   \Ex{x_k|\tilde
    x_{k-1},y_k} =AQ_{k-1}\tilde x_{k-1}+\underbrace{\tilde P_kC^*\big(C \tilde P_k C^*\! +R
\big)^{-1}}_{=:K_k}(y_k \! -CAQ_{k-1}\tilde x_{k-1}).
\end{equation}

It remains to compute the error covariance $P_k$ defined in
\eqref{eq:err_cov}, and the operator $Q_k$ defined through
\eqref{eq:Q_def}. By \eqref{eq:pythagoras}, $P_k$ is given by
\[
P_k=S_k-Q_k\tilde S_k Q_k^*
\]
where $S_k=\Cov{x_k,x_k}$ is the \emph{state covariance} and $\tilde
S_k=\Cov{\tilde x_k,\tilde x_k}$ is the \emph{state estimate
  covariance}. The state $x_k$ is a linear combination of mutually
independent Gaussian random variables $x_{k-1}$ and $u_k$ and so $S_k$
can be obtained from the Lyapunov difference equation
\begin{equation} \label{eq:S_k}
S_k=AS_{k-1}A^*+BUB^* 
\end{equation}
and the first one, $S_0$, is the initial state covariance in
\eqref{eq:system}. Also, by \eqref{eq:decomposition},
\begin{equation} \label{eq:y_k_comp}
y_k-CAQ_{k-1}\tilde x_{k-1}=CAv_k+CBu_k+w_k \sim N\big(0,C\tilde P_kC^*+R\big)
\end{equation}
where $v_k$, $u_k$, and $w_k$ are mutually independent and also
independent with the state estimate $\tilde x_{k-1}$. Thus, by
\eqref{eq:kf_proj}, also $\tilde S_k$ is obtained from a Lyapunov
difference equation,
\begin{equation} \label{eq:tildeS} \tilde S_k=\Pi AQ_{k-1}\tilde
  S_{k-1}Q_{k-1}^*A^*\Pi^*+\Pi K_k \big(C\tilde P_kC^*+R\big) (\Pi K_k)^T
\end{equation}
with $\tilde S_0=0$.

By \eqref{eq:cond}, $Q_k$ is given by
\begin{equation} \label{eq:Q_k} Q_k=
\Cov{x_k,\tilde x_k} \tilde S_k^{-1}.
\end{equation}
The case when $\tilde S_k$ is not invertible is discussed in
Remark~\ref{rem:Q_add}.
The cross covariance operator
$V_k:=\Cov{x_k,\tilde x_k}$ in \eqref{eq:Q_k} can be computed by
``anchoring'' $x_k$ and $\tilde x_k$ to $\tilde x_{k-1}$ using
equations \eqref{eq:decomposition} and \eqref{eq:kf_proj} and the fact that
$Av_{k-1}+Bu_k \sim N\big(0,\tilde P_k\big)$,
\begin{equation} \nonumber \Cov{x_k,\tilde x_k}=AQ_{k-1} \tilde
  S_{k-1} Q_{k-1}^*A^*\Pi^*+\tilde P_k C^*\big(C\tilde P_k
  C^*+R\big)^{-1}C\tilde P_k \Pi^*.
\end{equation}
It is worth noting here that $\tilde S_k=\Pi \Cov{x_k,\tilde x_k}$
implying the intuitive fact, $\Pi Q_k=I$ in the case that $\tilde S_k$
is invertible.

Let us conclude by presenting some remarks concerning the derivation
of the reduced-order state estimate and then collecting the relevant
equations to an algorithm.
\begin{remark} \label{rem:m_zero} \textup{The assumption $m=0$ does not
  restrict generality, since we can always always add $\Pi A^km$ to
  $\tilde x_k$ and subtract $CA^km$ from $y_k$ in
  \eqref{eq:decomposition}. However, this is how to make the
  derivation accurate. In practical implementation, it is reasonable
  to just start the state estimate from $\tilde x_0=\Pi m$ and then
  proceed as described.}
\end{remark}
\begin{remark} \label{rem:Q_add}
 \textup{ If $\tilde S_k$ is not invertible, it means that $\mathcal{R}(\tilde
  S_k)$, the range of $\tilde S_k$, does not cover the whole space
  $\mathbb{C}^n$.  The estimate $\tilde x_k$ lies on
  $\mathcal{R}(\tilde S_k)$ almost surely. Thus $Q_k$ is not
  determined uniquely in this case. By imposing additional
  requirements $\Pi Q_k=I$ and
  ${(I-\Pi_s)Q_k}\big|_{\mathcal{R}(\tilde S_k)^{\bot}}=0$ then $Q_k$
  is uniquely determined and it is given by $Q_k=\tilde
  Q_k+\Pi^*(I-\Pi \tilde Q_k)=\Pi^*+(I-\Pi_s)\tilde Q_k$ where $\tilde
  Q_k=\Cov{x_k,\tilde x_k} \tilde S_k^+$.}
\end{remark}

\begin{algorithm} \label{alg:method} \textup{As with the full state Kalman
  filter, the following operator-valued equations can be computed
  beforehand (offline):
\begin{align*}
  S_k &= AS_{k-1}A^*+BUB^*, \\
  \tilde P_k &=AP_{k-1}A^*+BUB^*, \\
  K_k &= \tilde P_kC^*\big(C \tilde P_k C^*\! +R
  \big)^{-1}, \\
V_k &= AQ_{k-1}\tilde
  S_{k-1}Q_{k-1}^*A^*\Pi^*+ K_k \big(C\tilde P_kC^*+R\big) (\Pi K_k)^T, \hspace{-6mm} \\
\tilde S_k&=\Pi V_k, \\
Q_k&=\Pi^*+(I-\Pi_s)\tilde S_k^+V_k, \\
P_k&=S_k-Q_k\tilde S_kQ_k^*.
\end{align*}
The initial values are $S_0$ (given in \eqref{eq:system}),
$P_0=S_0$, $\tilde S_0=0$, and $Q_0=\Pi^*$.
The state estimate is given by
\begin{align*}
\tilde x_0 & =\Pi m, \\
\tilde x_k &= \Pi AQ_{k-1}\tilde x_{k-1}+\Pi K_k(y_k-CAQ_{k-1}\tilde x_{k-1}). \hspace{7mm}
\end{align*}}
\end{algorithm}
\noindent Practical implementation of the proposed method is discussed
in Section~\ref{sec:practice}. An alternative equation for $P_k$ is
derived in the following section.

\section{The error covariance equation} \label{sec:asymptotics}


Motivated by the main theorem of \cite{Bernstein}, we next seek for a
Riccati difference equation satisfied by the error covariance
$P_k$. This equation will be needed later for determining a bound for
the error in the state estimate due to the spatial discretization. To
this end, define the augmented state $\bar x_k:=\bbm{x_k \\ \tilde x_k}$
for which we have dynamic equations
\begin{align*}  \! \bbm{x_k \\ \tilde x_k} \! &= \! \left[
    \!\! \begin{array}{cc} A & 0 \\ \Pi K_k C A &  \Pi( A
      - K_k C A) Q_{k-1} \end{array} \!\!\!  \right] \!\! \bbm{x_{k-1} \\
    \tilde x_{k-1}}+\left[ \!\! \begin{array}{cc} B & 0 \\ \Pi K_k CB
      & \Pi K_k \end{array} \!\!  \right] \!\! \bbm{u_k \\ w_k} \\
&=:\bar A_k \bar x_{k-1}+\bar B_k \bar
  u_k.
\end{align*}
The augmented state covariance satisfies the Lyapunov difference equation
\begin{equation} \label{eq:aug_cov}
\bar S_{k}=\bar A_k
\bar S_{k-1}\bar A_k^*+\bar B_k \bar U \bar B_k^*
\end{equation}
where $\bar U=\bbm{U & 0 \\ 0 & R}$. This covariance can be written
as a block operator by $\bar S_k=\bbm{S_k & V_k \\ V_k^* & \tilde
  S_k}$ where $S_k$ and $\tilde S_k$ are the state and state estimate
covariances, given in \eqref{eq:S_k} and \eqref{eq:tildeS},
respectively. Now it holds that $Q_k=V_k \tilde S_k^{-1}$ (or $Q_k=V_k
\tilde S_k^++\Pi^*(I-\Pi V_k \tilde S_k^+)$ if $\tilde S_k$ is not
invertible) and thus for the reduced-order error covariance defined in
\eqref{eq:err_cov}, it holds that $P_k=S_k-V_k \tilde
S_k^+V_k^*$.  Also, for the prediction error covariance we have
$\tilde P_k =A(S_k-V_k\tilde S_k^+V_k^*)A^*+BUB^*$ by
\eqref{eq:pred_cov}.
Using these notations we get from \eqref{eq:aug_cov}
\begin{align*}
  V_k &=AS_{k-1}A^*C^*K^*\Pi^*+AV_{k-1}\tilde S_{k-1}^+V_{k-1}^*A^*(\Pi-\Pi  K_kC)^*+BUB^*C^*K_k^*\Pi^* \\
  & =\tilde P_kC^*(C\tilde P_kC^*+R)^{-1}C\tilde
  P_k\Pi^*+AV_{k-1}\tilde S_{k-1}^+V_{k-1}^* A^*\Pi^*,
\end{align*}
and similarly $\tilde S_k=\Pi V_k=V_k^*\Pi^*$. Using the state
covariance Lyapunov equation \eqref{eq:S_k} and the equations above
and noting that $V_k\tilde S_k^+V_k^*=Q_kV_k^*=V_kQ_k^*=Q_k\tilde
S_kQ_k^*$, we see that the error covariance $P_k$ satisfies the
\emph{Riccati difference equation} (RDE)
\begin{equation} \label{eq:RDE_err}
\begin{cases}
  \tilde P_k=AP_{k-1}A^*+BUB^*, \\
  P_k=\tilde P_k-\tilde P_k C^*(C\tilde P_k C^*+R)^{-1}C\tilde P_k+ \\
  \  +(I-Q_k\Pi)(AV_{k-1}\tilde S_{k-1}^+V_{k-1}^*A^*\! +\tilde P_k
  C^*(C\tilde P_k C^*+R)^{-1}C\tilde P_k)(I-Q_k\Pi)^*.
\end{cases} \hspace{-6mm}
\end{equation}
This equation is posed in $\mathcal{L}(\sX)$. Note that this is not a
complete set of equations, but the last equation in
Algorithm~\ref{alg:method} can be replaced by the second equation in
\eqref{eq:RDE_err}. Compared to the RDE \eqref{eq:full_RDE} for the
full state Kalman filter, this equation contains the additional load
term in the last line of \eqref{eq:RDE_err}.  In the next section we
find an upper bound for the effect of this additional term to the
solution at the infinite time limit but first we need to go through
some auxiliary results.

\begin{proposition} \label{prop:cond}
  Let $\mathcal{S}_1$ and $\mathcal{S}_2$ be sigma algebras, such that
  $\mathcal{S}_1 \subset \mathcal{S}_2$ and $x$ an integrable random
  variable. Then
  $\Ex{x|\mathcal{S}_1}=\Ex{\Ex{x|\mathcal{S}_2}|\mathcal{S}_1}$.

  If $x$ is quadratically integrable then
\[  
\Cov{\Ex{x|\mathcal{S}_1},\Ex{x|\mathcal{S}_1}} \le
  \Cov{\Ex{x|\mathcal{S}_2},\Ex{x|\mathcal{S}_2}} \le \Cov{x,x}.
\]
\end{proposition}

\begin{lemma} \label{lem:M_k} Assume that the state covariance $S_k$
  defined in \eqref{eq:S_k} satisfies $S_k \le S$ for all $k$ for some
  trace class operator $S \in \mathcal{L}^*(\sX)$.  For the
  discretization error term in the RDE~\eqref{eq:RDE_err}, it holds
  that
\begin{align} \nonumber
M_k :=& \, (I-Q_k\Pi)(AV_{k-1}\tilde S_{k-1}^+V_{k-1}^*A^*\! +\tilde P_k
  C^*(C\tilde P_k C^*+R)^{-1}C\tilde P_k)(I-Q_k\Pi)^* \\
\label{eq:def_M}  \le& \, (I-\Pi_s)S(I-\Pi_s)^*=:M.
\end{align}
\end{lemma}
\begin{proof}
  Note that $V_{k-1}\tilde S_{k-1}^+V_{k-1}^*=Q_{k-1}\tilde
  S_{k-1}Q_{k-1}^*$. Then by \eqref{eq:kf_proj} and
  \eqref{eq:y_k_comp} it can be seen that
\begin{equation} \nonumber 
M_k=\Cov{(I-Q_k\Pi)\Ex{x_k|\tilde x_{k-1},y_k},(I-Q_k\Pi)\Ex{x_k|\tilde x_{k-1},y_k}}.
\end{equation}
It holds that
\begin{equation} \nonumber 
Q_k\Pi\Ex{x_k|\tilde x_{k-1},y_k}=Q_k\tilde x_k=\Ex{x_k|\tilde
   x_k}=\Ex{\Ex{x_k|\tilde x_{k-1},y_k}|\tilde x_k} 
\end{equation}
where the first equality follows by \eqref{eq:structure}, the second
by the definition of $Q_k$, \eqref{eq:Q_def}, and the third by
Proposition~\ref{prop:cond} and $\mathcal{S}(\tilde x_k) \subset
\mathcal{S}(\tilde x_{k-1},y_k)$ which, in turn, can be seen from
\eqref{eq:structure}.

Thus $Q_k$ minimizes
\begin{align*}
&\Ex{\ip{e,\Ex{x_k|\tilde x_{k-1},y_k}-Z\tilde x_k}_{\sX}^2}=\Ex{\ip{e,(I-Z\Pi)\Ex{x_k|\tilde x_{k-1},y_k}}_{\sX}^2} \\
&=\ip{e,(I-Z\Pi)\Cov{\Ex{x_k|\tilde x_{k-1},y_k},\Ex{x_k|\tilde x_{k-1},y_k}}(I-Z\Pi)^*e}_{\sX}
\end{align*}
over $Z\in
\mathcal{L}(\mathbb{C}^n,\sX)$ for all $e \in \sX$.
Since $\Pi_s=\Pi^*\Pi$, it holds that 
\begin{align*}
  M_k &\le (I-\Pi_s)\Cov{\Ex{x_k|\tilde x_{k-1},y_k},\Ex{x_k|\tilde
      x_{k-1},y_k}}(I-\Pi_s)^* \\
  &\le (I-\Pi_s)\Cov{x_k,x_k}(I-\Pi_s)^* \le M
\end{align*}
where the middle inequality holds by Proposition~\ref{prop:cond}.
\end{proof}
\begin{lemma} \label{lem:mono}
Let $P_k^{(j)}$ for $j=1,2$, be the solutions of the RDEs
\begin{equation} \label{eq:RDE}
\begin{cases}
  \tilde P_k^{(j)}=AP_{k-1}^{(j)}A^*+W_k^{(j)}, \\
  P_k^{(j)}=\tilde P_k^{(j)}-\tilde P_k^{(j)} C^*(C\tilde
  P_k^{(j)}C^*+R)^{-1}C\tilde P_k^{(j)}
\end{cases}
\end{equation}
where $P_0^{(2)} \ge P_0^{(1)} \ge 0$ and $W_k^{(2)} \ge W_k^{(1)} \ge
0$. Then $P_k^{(2)} \ge P_k^{(1)}$ for all $k \ge 0$.
\end{lemma}
\noindent This follows from \cite[Lemma~3.1]{deSouza_mono} by de Souza
in the finite dimensional setting. The proof is just algebraic
manipulation and it holds also in the infinite dimensional setting (if
the output is finite dimensional). However, we shall present a 
straightforward proof.
\begin{proof}
  We show $P_1^{(2)} \ge P_1^{(1)}$. For larger $k$ the result follows by
  induction. Define the block diagonal covariances in $\mathcal{L}^*(\sX^3)$
\[
\tilde P_B^{(1)}=\sbm{AP_0^{(1)}A^* & & \\ & \hspace{-1mm} W_1^{(1)} &
  \\ & & \hspace{-0mm} 0} \qquad \textrm{and} \qquad \tilde
P_B^{(2)}=\sbm{AP_0^{(2)}A^* & & \\ & \hspace{-1mm} W_1^{(1)} & \\ & &
  \hspace{-1mm} W_1^{(2)}\!-W_1^{(1)}}
\]
and $C_B:=[C \ C \ C]$. Then define
\begin{align*}
  P_B^{(j)}&=\tilde P_B^{(j)}-\tilde P_B^{(j)}C_B^*(C_B\tilde P_B^{(j)}C_B^*+R)^{-1}C_B\tilde P_B^{(j)} \qquad \textrm{for } j=1,2 \\
  P_B^{(\times)}&=\tilde P_B^{(2)}-\tilde P_B^{(2)}C_B^*(C_B\tilde
  P_B^{(1)}C_B^*+R)^{-1}C_B\tilde P_B^{(2)}.
\end{align*}
Now $\tilde P_B^{(2)} \ge \tilde P_B^{(1)}$ implies $P_B^{(2)} \ge
P_B^{(\times)}$. Then $P_B^{(1)}=\sbm{I & & \\ & I & \\ & &
  0}P_B^{(\times)}\sbm{I & & \\ & I & \\ & & 0}$ and so $P_B^{(\times)} \ge
P_B^{(1)}$. Now $P_1^{(j)}=[I \ I \ I]P_B^{(j)}\sbm{I \\ I \\ I}$ and
so $P_1^{(2)} \ge P_1^{(1)}$.
\end{proof}

The following lemma is due to Hager and Horowitz,
\cite{Hager_Horowitz}:
\begin{lemma} \label{lem:RDE_conv} Assume that $S_k \le S$ for all $k$
  for some trace class operator $S \in \mathcal{L}^*(\sX)$ where $S_k$
  is defined in \eqref{eq:S_k}. Let $P_k^{(F)}$ be the solution of
  \eqref{eq:full_RDE} and $P_k^{(b)}$ be the solution of
  \eqref{eq:RDE} with $W_k^{(b)}=W^{(b)}=BUB^*+AMA^*$ where $M$ is
  defined in \eqref{eq:def_M}.  Assuming $P_0^{(b)}=P_0^{(F)}=0$, then
  $P_k^{(b/F)} \to P^{(b/F)}$ strongly as $k \to \infty$.  Also, the
  limit operators $P^{(b/F)} \ge 0$ are the unique nonnegative
  solutions of the discrete time algebraic Riccati equation (DARE)
\begin{equation} \label{eq:DARE}
\begin{cases}
\tilde P^{(b/F)}=AP^{(b/F)}A^*+W^{(b/F)}, \\
  P^{(b/F)}=\tilde P^{(b/F)}-\tilde P C^*(C\tilde
  P^{(b/F)}C^*+R)^{-1}C\tilde P^{(b/F)}
\end{cases}
\end{equation}
where $W^{(F)}=BUB^*$.

If $\sigma(A- K^{(F)} CA) \subset B(0,\rho)$ with $\rho < 1$ where
$K^{(F)}$ is the limit of the full state Kalman gain, that is
\begin{equation} \label{eq:K_inf}
K^{(F)}=\tilde P^{(F)}C^*(C\tilde P^{(F)}C^*+R)^{-1}, 
\end{equation}
then
$P_k^{(F)} \to P^{(F)}$ strongly, starting from any $P_0^{(F)} \ge 0$.
\end{lemma}
\noindent The first part follows from \cite[Theorem~1]{Hager_Horowitz}
because $P_k^{(j)} \le S$, and the second part from
\cite[Theorem~3]{Hager_Horowitz}.

Even the weak convergence would suffice for the dominated convergence
of trace class operators:
\begin{lemma} \label{lem:tr_conv} If $P$, $S$, and $P_k$ for
  $k=0,1,...$ are trace class operators in $\mathcal{L}^*(\sX)$, $P_k
  \le S$ for all $k$, and $P_k \overset{\textrm{w}}{\longrightarrow}
  P$, then $\tr(P_k) \to \tr(P)$.
\end{lemma}
\noindent The proof is rather straightforward after noting that
$\ip{e_j,P_ke_j}_{\sX} \to \ip{e_j,Pe_j}_{\sX}$ as $k \to \infty$, for
all $j \in \mathbb{N}$ where $\{e_j\}_{j \in \mathbb{N}}$ is an
orthonormal basis for $\sX$.


\section{Error analysis} \label{sec:error}

Next we use sensitivity analysis for DAREs and the results of the
preceding section to show a bound for the discrepancy
$\Ex{\norm{Q_k\tilde x_k- \hat{x}_k}_{\sX}^2}$ of the full and
reduced-order state estimates, defined in \eqref{eq:KF} and
\eqref{eq:structure}, respectively. The results of this section are
based on bounding the effect of the perturbation $M_k$ in
\eqref{eq:def_M} caused by the spatial discretization. Such bound is
possible if we have additional information about the smoothness of the
state $x_k$. That is, it is assumed that $x_k$ lies in a subspace
$\sX_1$ of $\sX$ --- which is a Hilbert space itself --- and that the
projection $\Pi_s$ approximates well the vectors in that subspace,
meaning that the norm $\norm{I-\Pi_s}_{\mathcal{L}(\sX_1,\sX)}$
becomes small as the spatial discretization is refined.

We show two theorems --- first (Thm.~\ref{thm:error}) is
an \emph{a priori} type estimate on the convergence rate of
$\Ex{\norm{Q_k\tilde x_k- \hat{x}_k}_{\sX}^2}$, and the second
(Thm.~\ref{thm:posteriori}) is an \emph{a posteriori} estimate of the
error $\Ex{\norm{Q_k\tilde x_k- \hat{x}_k}_{\sX}^2}$.
\begin{theorem} \label{thm:error} 
  Consider the system \eqref{eq:system} and the reduced order state
  estimator $Q_k\tilde x_k$ derived in Sections~\ref{sec:estimate} and
  \ref{sec:asymptotics}. Make the following assumptions:
\begin{itemize}

\item[(i)] $x_k \in \sX_1$ a.s. for all $k$ where $\sX_1$ is a Hilbert
  space that is a vector subspace of $\sX$ and
  $\sup_k\Ex{\norm{x_k}_{\sX_1}^2} < \infty$.



\item[(ii)] The state covariance $S_k$ defined in \eqref{eq:S_k}
  converges to the solution of the Lyapunov equation $S=ASA^*+BUB^*$,
  that is, $S=\sum_{j=0}^{\infty}A^jBUB^*(A^*)^j$ and $S_k \le S$ for
  all $k \ge 0$. Use this $S$ in the definition of $M$ in
  \eqref{eq:def_M}.

  
\item[(iii)] The converged full state Kalman filter is exponentially
  stable, meaning $\sigma(A-K^{(F)}CA) \subset B(0,\rho)$ for some
  $\rho < 1$ where $K^{(F)}$ is the Kalman gain of the converged full
  state Kalman filter, introduced in \eqref{eq:K_inf}.

\end{itemize}

\noindent If $\norm{I-\Pi_s}_{\mathcal{L}(\sX_1,\sX)}$ is small enough, it
holds that
\begin{equation} \nonumber \limsup_{k \to \infty}\Ex{\norm{Q_k\tilde x_k-
      \hat{x}_k}_\sX^2} \le C
  \norm{I-\Pi_s}_{\mathcal{L}(\sX_1,\sX)}^2 + \mathcal{O} \! \left(
    \norm{I-\Pi_s}_{\mathcal{L}(\sX_1,\sX)}^4 \right)
\end{equation}
where $\displaystyle C=\left(\! 1+L\norm{A-
    K^{(F)}CA}_{\mathcal{L}(\sX)}^2\right)\sup_k
\Ex{\norm{x_k}_{\sX_1}^2}$
and $L$ is defined in Lemma~\ref{lem:L_properties}.


\end{theorem}
\begin{proof} 
  

  Assume first that the initial state is completely known, that is,
  $S_0=0$.  Let $P_k$ be the error covariance of the reduced order
  method, satisfying the RDE~\eqref{eq:RDE_err} and $M_k$ be defined
  in \eqref{eq:def_M}. It is easy to confirm that the \emph{shifted
    covariance} $P_k^{(a)}:=P_k-M_k$ satisfies the RDE
\begin{equation} \nonumber
\begin{cases}
  \tilde P_k^{(a)}=AP_{k-1}^{(a)}A^*+BUB^*+AM_kA^*, \\
  P_k^{(a)}=\tilde P_k^{(a)}-\tilde P_k^{(a)} C^*\big(C\tilde
  P_k^{(a)}C^*+R \big)^{-1}C\tilde P_k^{(a)}.
\end{cases}
\end{equation}
Then denote by $P_k^{(b)}$ and $\tilde P_k^{(b)}$ the solution of a
similar RDE but with the term $AM_kA^*$ replaced by $AMA^*$ where $M$
is the upper bound for $M_k$, defined in \eqref{eq:def_M}.  Finally,
let $P_k^{(F)}$ be the error covariance of the full Kalman filter
estimate, given in \eqref{eq:full_RDE} and $\hat{x}_k=\Ex{x_k|\{y_j,j
  \le k\}}$ is given in \eqref{eq:KF}.


By computing the trace of both sides of \eqref{eq:pythagoras}, we see
that for a Gaussian random variable $[h_1,h_2]$ it holds that
\[
\Ex{\norm{h_1}_{\sX}^2}=\Ex{\norm{\Ex{h_1|h_2}}_{\sX}^2}+\Ex{\norm{h_1-\Ex{h_1|h_2}}_{\sX}^2}.
\]
Now $\tilde x_k$ depends linearly on $[y_1,...,y_k]$ and thus clearly
$\mathcal{S}( \tilde x_k) \subset \mathcal{S}(y_1,...,y_k)$. By
Proposition~\ref{prop:cond}, it holds that $Q_k\tilde
x_k=\Ex{x_k|\tilde x_k}=\Ex{\hat x_k|\tilde x_k}$. Thus it holds that
\begin{align*} & \hspace{-5mm}\Ex{\norm{Q_k \tilde x_k-\hat
      x_k}_{\sX}^2}=\Ex{\norm{\hat x_k}_{\sX}^2}-\Ex{\norm{Q_k \tilde x_k}_{\sX}^2} \\
  \hspace{5mm} &=\Ex{\norm{x_k}_{\sX}^2}-\Ex{\norm{Q_k \tilde x_k}_{\sX}^2}-
  \left(\Ex{\norm{x_k}_{\sX}^2}-\Ex{\norm{\hat x_k}_{\sX}^2}
  \right) \\
  &=\Ex{\norm{x_k-Q_k \tilde x_k}_{\sX}^2}-\Ex{\norm{x_k-\hat
      x_k}_{\sX}^2}=\tr\big(P_k^{(a)} \big)+\tr(M_k)-\tr\big(P_k^{(F)}
  \big). 
\end{align*}
By Lemmas~\ref{lem:M_k} and \ref{lem:mono}, $P_k^{(F)} \le P_k^{(a)}
\le P_k^{(b)}$ and thus $\tr \big(P_k^{(a)}\big)-\tr
\big(P_k^{(F)}\big) \le \tr
\big(P_k^{(b)}\big)-\tr\big(P_k^{(F)}\big)$. By
Lemma~\ref{lem:RDE_conv}, $P_k^{(b)} \to P^{(b)}$ and $P_k^{(F)} \to
P^{(F)}$ strongly (recall $S_0=0$) where $P^{(b)}$ and $P^{(F)}$ are
the solutions of the corresponding DAREs, that is,
equation~\eqref{eq:DARE} with $W^{(b)}=BUB^*+AMA^*$ and
$W^{(F)}=BUB^*$. Also, by Lemma~\ref{lem:tr_conv}, $\tr\big(P_k^{(b)}
\big) \to \tr\big(P^{(b)}\big)$ and $\tr\big(P_k^{(F)} \big) \to
\tr\big(P^{(F)}\big)$. Denote $\Delta P:=P^{(b)}-P^{(F)}$ and note
that $\Delta P \in \mathcal{L}^*(\sX)$ is a positive (semi-)definite
trace class operator. Then an upper bound for the discrepancy is given
by
\begin{equation} \label{eq:limsup}
\limsup_{k \to \infty}\Ex{\norm{Q_k \tilde x_k-\hat x_k}_{\sX}^2} \le
\tr(\Delta P)+\tr(M).
\end{equation}

Equation~\eqref{eq:DP} in Lemma \ref{lem:DP} gives a representation
for $\Delta P$. The next step is to use this equation to find a bound
for $\tr(\Delta P)$.
Because the full Kalman filter is assumed to be exponentially stable,
by Lemmas~\ref{lem:L_properties} and~\ref{lem:DP}, we have
\begin{equation} \nonumber
\tr(\Delta P) \le \tr \big({\bf L}^{-1}(E_1+E_2+h_1(\Delta P)) \big)
\end{equation}
where ${\bf L} \in \mathcal{L}(\mathcal{L}^*(\sX))$ is defined in
Lemma~\ref{lem:L_properties} and $E_1$, $E_2$, and $h_1(\Delta P)$ are
defined in Lemma~\ref{lem:DP}. The term $h_2(\Delta P)$ in
\eqref{eq:DP} is excluded here because it is negative definite (see
the discussion after Lemma~\ref{lem:L_properties}).

Now we have $E_1 \ge 0$ and so by Lemma~\ref{lem:L_properties},
\begin{equation} \nonumber \tr({\bf L}^{-1}E_1) \le L\tr(E_1) \le
  L\norm{A- K^{(F)}CA}_{\mathcal{L}(\sX)}^2 \tr(M)
\end{equation}
where $L$ is defined in Lemma~\ref{lem:L_properties}. From $E_2$ the
negative definite part can be omitted and thus
\begin{align*} & \hspace{-5mm}\tr({\bf L}^{-1}E_2) \le L\norm{
    K^{(F)}C}_{\mathcal{L}(\sX)}^2 \! \tr \! \left( \! AMA^*C^*\left( C(\tilde
      P^{(F)} \! +AMA^*)C^* \! +R \right)^{-1} \! CAMA^* \! \right) \\
  \hspace{5mm}& \le L\norm{K^{(F)}C}_{\mathcal{L}(\sX)}^2
  \norm{A}_{\mathcal{L}(\sX)}^4\norm{C}_{\mathcal{L}(\sX,\sY)}^2\tr \!
  \left( \!\! \left(
      C(\tilde P^{(F)}+AMA^*)C^* \! + R \right)^{-1}\right)\tr(M)^2 \\
  & \le L\norm{K^{(F)}C}_{\mathcal{L}(\sX)}^2
  \norm{A}_{\mathcal{L}(\sX)}^4\norm{C}_{\mathcal{L}(\sX,\sY)}^2\tr \!
  \left( \!\! \left( C\tilde P^{(F)}C^* \! +R
    \right)^{-1}\right)\tr(M)^2.
\end{align*}
To get a bound for $\tr\left({\bf L}^{-1}h_1(\Delta P) \right)$,
recall the following properties of the operator trace and the
Hilbert-Schmidt norm:
\begin{equation} \nonumber \norm{AB}_{HS} \le
  \norm{A}_{\mathcal{L}(\sX)}\norm{B}_{HS}, \quad \norm{A}_{HS} \le \tr(A),
  \textrm{ for } A \in \mathcal{L}^*(\sX), \ A \ge 0, 
\end{equation}
\begin{equation} \nonumber
\textrm{and} \quad
\tr(AB) \le \norm{A}_{HS}\norm{B}_{HS}.
\end{equation}
Using these and \eqref{eq:L_inv} yields
$\tr\left({\bf L}^{-1}h_1(\Delta P) \right)$
\begin{equation} \nonumber
 \le L_0 \left(2\norm{A-
      K^{(F)}CA}_{\mathcal{L}(\sX)}\norm{CA}_{\mathcal{L}(\sX,\sY)}\norm{\Delta
      K}_{HS} +\norm{CA}_{\mathcal{L}(\sX,\sY)}^2 \norm{\Delta K}_{HS}^2 \right)
  \tr(\Delta P)
\end{equation}
where $L_0$ is defined in Lemma~\ref{lem:L_properties}, $\Delta K=
K^{(F)}- K^{(b)}$, and $K^{(b)}=\tilde P^{(b)}C^*(C\tilde P^{(b)}C^*+R)^{-1}$.
By the last part of Lemma~\ref{lem:DP}, we have
\begin{equation} \label{eq:DK}
\norm{\Delta K}_{HS} \le \big(\hat c_1+ \hat c_2 \tr(M) \big)\tr(M)
\end{equation}
where
\begin{align*}  \hat c_1= &\left(1+ \norm{\tilde
      P^{(F)}}_{\mathcal{L}(\sX)}\norm{C}_{\mathcal{L}(\sX,\sY)}^2\norm{\left(
        C\tilde P^{(F)}C^*+R \right)^{-1}}_{\mathcal{L}(\sY)}\right) \times \\
& \qquad \times
  \norm{C}_{\mathcal{L}(\sX,\sY)}\norm{A}_{\mathcal{L}(\sX)}^2\norm{\left(
      C\tilde P^{(F)}C^*+R \right)^{-1}}_{\mathcal{L}(\sY)}
\end{align*}
\begin{equation} \nonumber \hspace{-20mm} \textrm{and} \hspace{21mm}
  \hat c_2=\norm{A}_{\mathcal{L}(\sX)}^4\norm{C}_{\mathcal{L}(\sX,\sY)}^3\norm{\left(
      C\tilde P^{(F)}C^*+R \right)^{-1}}_{\mathcal{L}(\sY)}^2.
\end{equation}
Collecting these inequalities we finally get
\begin{equation} \label{eq:tr_DP} \tr(\Delta P) \le \frac{a
    \tr(M)+b\tr(M)^2}{1-\big(c_1\tr(M)+c_2\tr(M)^2+c_3\tr(M)^3+c_4\tr(M)^4\big)}
\end{equation}
where
\begin{align*}
  a&=L\norm{A-K^{(F)}CA}_{\mathcal{L}(\sX)}^2, \\
  b&=L\norm{K^{(F)}C}_{\mathcal{L}(\sX)}^2
  \norm{A}_{\mathcal{L}(X)}^4\norm{C}_{\mathcal{L}(\sX,\sY)}^2\tr
  \left( \! \left(
      C\tilde P^{(F)}C^*+R \right)^{-1}\right), \\
  c_1&=2L_0\norm{A-
    K^{(F)}CA}_{\mathcal{L}(\sX)}\norm{CA}_{\mathcal{L}(\sX,\sY)}\hat c_1, \\
  c_2&=2L_0\norm{A-
    K^{(F)}CA}_{\mathcal{L}(\sX)}\norm{CA}_{\mathcal{L}(\sX,\sY)}\hat c_2+L_0\norm{CA}_{\mathcal{L}(\sX,\sY)}^2\hat c_1^2, \\
  c_3&=2L_0\norm{CA}_{\mathcal{L}(\sX,\sY)}^2 \hat c_1 \hat c_2, \\
  c_4&=L_0\norm{CA}_{\mathcal{L}(\sX,\sY)}^2 \hat c_2^2.
\end{align*}

To complete the proof under the assumption $S_0=0$, use
\eqref{eq:limsup}, \eqref{eq:tr_DP}, and note that by the definition
of $M$ in \eqref{eq:def_M} and $S$ in assumption {\it
  (ii)}, 
\begin{equation} \nonumber \tr(M) =\sup_k
  \Ex{\norm{(I-\Pi_s)x_k}_{\sX}^2} \le
  \norm{I-\Pi_s}_{\mathcal{L}(\sX_1,\sX)}^2\sup_k
  \Ex{\norm{x_k}_{\sX_1}^2}.  \vspace{-3mm}
\end{equation}

In case $S_0 > 0$, the convergence $P_k^{(b)} \to P^{(b)}$ has to be
established. Denote $\Phi=A-K^{(F)}CA$ and $\Delta \Phi=\Delta
KCA$. Pick $\lambda \in \mathbb{C}$ from the resolvent set of
$\Phi$. Then using the Woodbury formula, we get
\begin{align*}
  &\left(\lambda-(A-K^{(b)}CA)\right)^{-1}=\left(\lambda-\Phi-\Delta
    \Phi \right)^{-1} \\
  =&(\lambda-\Phi)^{-1}+(\lambda-\Phi)^{-1}\Delta \Phi \left(
    I-(\lambda-\Phi)^{-1}\Delta \Phi \right)^{-1}(\lambda-\Phi)^{-1}
\end{align*}
and
$\hspace{25mm} \displaystyle \norm{(\lambda-\Phi)^{-1}\Delta \Phi}_{\mathcal{L}(\sX)} \le
\frac{\norm{\Delta \Phi}_{\mathcal{L}(\sX)}}{|\lambda | - \rho}$ \newline
where $\rho<1$ is the spectral radius of $\Phi$.  The invertibility of
$\lambda-(A-K^{(b)}CA)$ is then guaranteed if $\norm{\Delta
  KCA}_{\mathcal{L}(\sX)} < |\lambda |-\rho$ which implies that the
spectral radius of $A-K^{(b)}CA$ is at most $\rho+\norm{\Delta
  KCA}_{\mathcal{L}(\sX)}$. So when $\tr(M)$ is small enough, then
also $A-K^{(b)}CA$ is exponentially stable and $P_k^{(b)} \to P^{(b)}$
strongly.
\end{proof}


The assumption {\it (iii)} in Theorem~\ref{thm:error} is very
difficult to check. Also, it is hard to say what it means that
``$\norm{I-\Pi_s}_{\mathcal{L}(\sX_1,\sX)}$ {\it is small enough}''
which is related to the denominator in Eq.~\eqref{eq:tr_DP} and the
exponential stability of $A-K^{(b)}CA$.
Consequently, this theorem should be considered as an \emph{a priori}
convergence speed estimate when the discretization is refined, that
is, when $\norm{I-\Pi_s}_{\mathcal{L}(\sX_1,\sX)} \to 0$.

However, if one has already computed the operators $Q_k$ and $K_k$ and
they have converged to $Q_{\infty}$ and $K_{\infty}$ and it has turned
out that $\sigma(A-K_{\infty}CA) \subset B(0,\rho)$ for some $\rho <
1$, then by the same argument as in Theorem~\ref{thm:error} we get the
following improved error estimate:
\begin{theorem} \label{thm:posteriori}
  Make the assumptions (i) and (ii) in Theorem~\ref{thm:error}.
  Assume also that the operators $K_k$, $Q_k$, and $M_k$ related to
  the reduced order filter have converged to $K_{\infty}$, $Q_{\infty}$ and
  $M_{\infty}$, respectively, and $\sigma(A-K_{\infty}CA) \subset B(0,\rho)$
  for some $\rho < 1$.
Then
\begin{equation} \nonumber \limsup_{k \to \infty}\Ex{\norm{Q_k\tilde
      x_k- \hat{x}_k}_{\sX}^2} \le C_1
  \norm{I-\Pi_s}_{\mathcal{L}(\sX_1,\sX)}^2+C_2\norm{I-\Pi_s}_{\mathcal{L}(\sX_1,\sX)}^4
\end{equation}
where $\displaystyle C_1=\left(1+\tilde L \norm{A-
    K^{(F)}CA}_{\mathcal{L}(\sX)}^2 \right)\sup_k \Ex{\norm{x_k}_{\sX_1}^2}$,
\[
\displaystyle C_2=\tilde L \norm{K^{(F)}C}_{\mathcal{L}(\sX)}^2
\norm{A}_{\mathcal{L}(\sX)}^4\norm{C}_{\mathcal{L}(\sX,\sY)}^2\tr \!
\left( \!\! \left( C\tilde P^{(F)}C^* \!+R
  \right)^{\! -1}\right) \!\! \left(\sup_k \Ex{\norm{x_k}_{\sX_1}^2 \! \right)^{\! 2}},
\]
 and $\tilde L$ is defined in
Lemma~\ref{lem:L_properties}.
\end{theorem}
\begin{proof}
  The covariances $P_k^{(a)}$ and $\tilde P_k^{(a)}$ defined in the
  proof of Theorem~\ref{thm:error} converge to $P^{(a)}$ and $\tilde
  P^{(a)}$ that are the solution of the DARE
\[
\begin{cases}
 \tilde P^{(a)}=AP^{(a)}A^*+BUB^*+AM_{\infty}A^*, \\
  P^{(a)}=\tilde P^{(a)}-\tilde P^{(a)} C^*\big(C\tilde
  P^{(a)}C^*+R \big)^{-1}C\tilde P^{(a)}.
\end{cases}
\]
Now bounding $\Delta P:=P^{(a)}-P^{(F)}$ by using the alternative
expression \eqref{eq:DP_alt} for $\Delta P$ given in
Lemma~\ref{lem:DP} and otherwise proceeding as in the proof of
Theorem~\ref{thm:error} leads to the result. Note that 
\[
K_{\infty}=\lim_{k \to \infty}\tilde P_kC^*(C\tilde P_kC^*+R)^{-1}
\]
but since $\tilde P_k^{(a)}=\tilde P_k$ for all $k$, it holds that
$K_{\infty}=\tilde P^{(a)}C^*(C\tilde P^{(a)}C^*+R)^{-1}$.
\end{proof}


\begin{remark} \label{rem:constants} \textup{The coefficients $C_1$ and $C_2$
  in the above theorem depend on $K^{(F)}$ and $\tilde P^{(F)}$ which
  is not desirable. It is possible to bound these coefficients from
  above without computing them. Firstly, we have
  \begin{equation} \nonumber \norm{A-K^{(F)}CA}_{\mathcal{L}(\sX)}^2
    \le
    2\norm{A-K_{\infty}CA}_{\mathcal{L}(\sX)}^2+2\norm{CA}_{\mathcal{L}(\sX,\sY)}^2\norm{\Delta
      K}_{\mathcal{L}(\sY,\sX)}^2.
\end{equation}
Now $\norm{\Delta K}_{\mathcal{L}(\sY,\sX)} \le \norm{\Delta K}_{HS}$ for
which we have \eqref{eq:DK}, $\norm{\tilde
  P^{(F)}}_{\mathcal{L}(\sX)} \le \norm{\tilde P^{(a)}}_{\mathcal{L}(\sX)}$, 
$\norm{\left( C\tilde P^{(F)}C^*+R
    \right)^{-1}}_{\mathcal{L}(\sY)} \! \le \frac1{\min(\textup{eig}(R))}, \ $
  and $\ \tr \!
\left(\!\! \left( C\tilde P^{(F)}C^*+R
  \right)^{-1}\right) \le \tr(R^{-1})$.}
\end{remark}

\section{Numerical example} \label{sec:example}

In this section, Algorithm~\ref{alg:method} is implemented to the
temporally discretized 1D wave equation with damping,
\begin{equation} \label{eq:wave}
\begin{cases}
  \frac{\partial^2}{\partial t^2}z(x,t)=-\epsilon\frac{\partial}{\partial t}z(x,t)+\frac{\partial^2}{\partial x^2}z(x,t) + Bu(t), \qquad x \in [0,1], \\
  z(0,t)=z(1,t)=0, \\
  y(t)=Cz(x,t) + w(t), \\
  z(x,0)=z_0
\end{cases}
\end{equation}
where $u \in \mathbb{R}^3$ and $w \in \mathbb{R}^2$ are the formal
derivatives of Brownian motions with incremental covariances $U$ and
$R$, respectively. The initial state is a Gaussian random variable
$z_0 \sim N(0,P_0)$, and $u$, $w$, and $z_0$ are mutually
independent. The input operator $B$ is a multiplication operator but
we define its structure only on the discrete-time level. The output
operator $C \in \mathcal{L}(\sX,\mathbb{R}^2)$ is given by
$Cz=\left[\ip{c_1,z}_{L^2(0,1)},\ip{c_2,z}_{L^2(0,1)} \right]^T$ where
$c_1(x)=\frac{1.4}{(x+1)^{.7}}$ and $c_2(x)=\frac1{(2-x)^{.3}}$.

The equation is transformed to a first order differential equation
with respect to the time variable by introducing the augmented state
$\bbm{z \\ v}$ where $v=\frac{\partial}{\partial t}z$ is the velocity
variable. The natural augmented state space is $\sX=H_0^1[0,1] \times
L^2(0,1)$. In $H_0^1[0,1]:=\{ z \in H^1[0,1] \, | \, z(0)=z(1)=0 \}$
we use the norm $\norm{z}_{H_0^1[0,1]}^2:=\int_0^1z'(x)^2 dx$.  The
equation is then temporally discretized using the implicit Euler
method with time step $\Delta t$.  The state space discretization is
carried out by Finite Element Method using piecewise linear elements
on two meshes on the interval $[0,1]$. The first one is a finer mesh
with $N_f$ equispaced discretization points. The fine mesh solution is
regarded as the true solution. The second, coarse mesh consists of
$N_c$ discretization points, also equally spaced with discretization
intervals of length $h_c=1/(N_c+1)$. It is required that the function
space consisting of the piecewise linear elements on the coarse mesh
is a subspace of the fine mesh space. This is satisfied when $N_f+1 =
k(N_c + 1)$ for some integer $k$. The coarse mesh space is the range
of $\Pi$. In the augmented state of the discretized system, the input
operator is $B_d=\bbm{0 & 0 & 0 \\ b_1(x) & b_2(x) & b_3(x)}$ where
$b_1(x)=(1-x)\sin(\pi x)$, $b_2(x)=7x^2(1-x)$, and $b_3(x)=\sin(6\pi
x)^2/x$. The input noise covariance for the discrete time system is
$U_d=\Delta t U$.

The solution of \eqref{eq:wave} actually has additional smoothness,
namely $[z \ v]^T \in \sX_1 = \left(H_0^1[0,1] \cap H^2[0,1] \right)
\times H_0^1[0,1]$ almost surely --- note that $B_d \in
\mathcal{L}(\mathbb{R}^3,\sX_1)$. It is well known that the piecewise
linear elements approximate $H^2$-functions in one dimension so that
$\norm{z-\Pi_s z}_{H^1[0,1]} \le C_2 h_c\norm{z}_{H^2[0,1]}$ and
$H^1$-functions so that $\norm{v-\Pi_s v}_{L^2(0,1)} \le C_1
h_c\norm{v}_{H^1[0,1]}$, see for example \cite[Section~5.1]{vidar}.

Fig.~\ref{fig:example} (left) shows the state $z(x,t)$ together with
the three different state estimates in one simulation. The full state
Kalman filter estimate (F) and the reduced-order state estimate (A)
cannot be distinguished from each other. The third state estimate (C)
is computed in the coarse mesh without taking the discretization error
into account. The simulation parameters are shown in
Table~\ref{tab:parameters} (left). The spectral radius was .996 for
both the full state Kalman filter and the reduced-order filter.  We
are interested in the stationary Kalman filter and so the
simulations were first run 2000 steps to get rid of initial
transitions. The expected (squared) errors of the different methods
are shown in Table~\ref{tab:parameters} (right) separately for the
position variable $z$ and the velocity variable $v$.


As $h_c \to 0$, the expected squared difference between the
reduced-order estimate and full state Kalman filter estimate, $\lim_{k
  \to \infty}\Ex{\norm{Q_k\tilde x_k- \hat{x}_k}_{\sX}^2}$, tends to
zero. Fig.~\ref{fig:example} (right) illustrates this convergence in
the example case. Regression analysis gives $\lim_{k \to
  \infty}\Ex{\norm{Q_k\tilde x_k- \hat{x}_k}_{\sX}^2} \approx 86.8h^{7.06}$
whereas Theorem~\ref{thm:error} gives $\mathcal{O}(h^2)$ convergence
rate.

\begin{table}[t]
  \centering
\vspace{-2mm}
\caption{Left: Simulation parameters. Right: Squared error averages over 500 simulations.}
    \vspace{2mm}
  \begin{tabular}{cl}
    \hline
    \hline
 Symbol & value  \\
\hline 
$\Delta t$ & .01 \\
$U$ & diag(1 , 1 , .25) \\
$R$ & diag(.3 , .15) \\
$N_f$ & 65 \\
$N_c$ & 5 \\
$\epsilon$ & .4 \\
   \hline
    \hline
  \end{tabular}
\hspace{10mm}
\begin{tabular}{lccc}
    \hline
    \hline
    Method  & F & A &  C  \\
    \hline 
 Position & .6122 & .6126 & .6352  \\
    Velocity & .8150 & .8154 & .9294  \\
   \hline
    \hline
\vspace{12.65mm}
  \end{tabular}
  \label{tab:parameters}
\end{table}

\begin{figure} [t]
  \centering
\hspace{1mm}
\includegraphics[width=5.5cm]{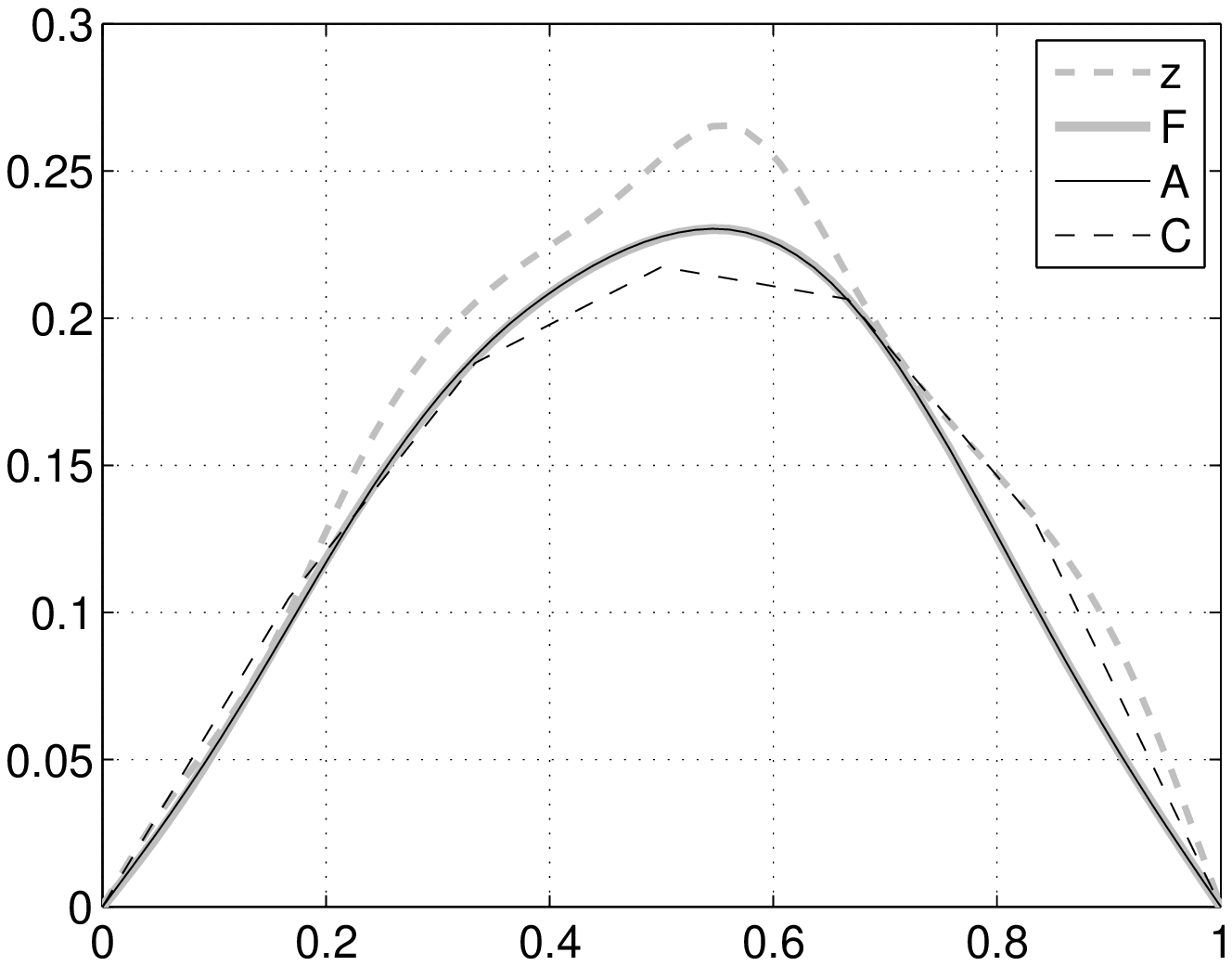} 
\includegraphics[width=5.5cm]{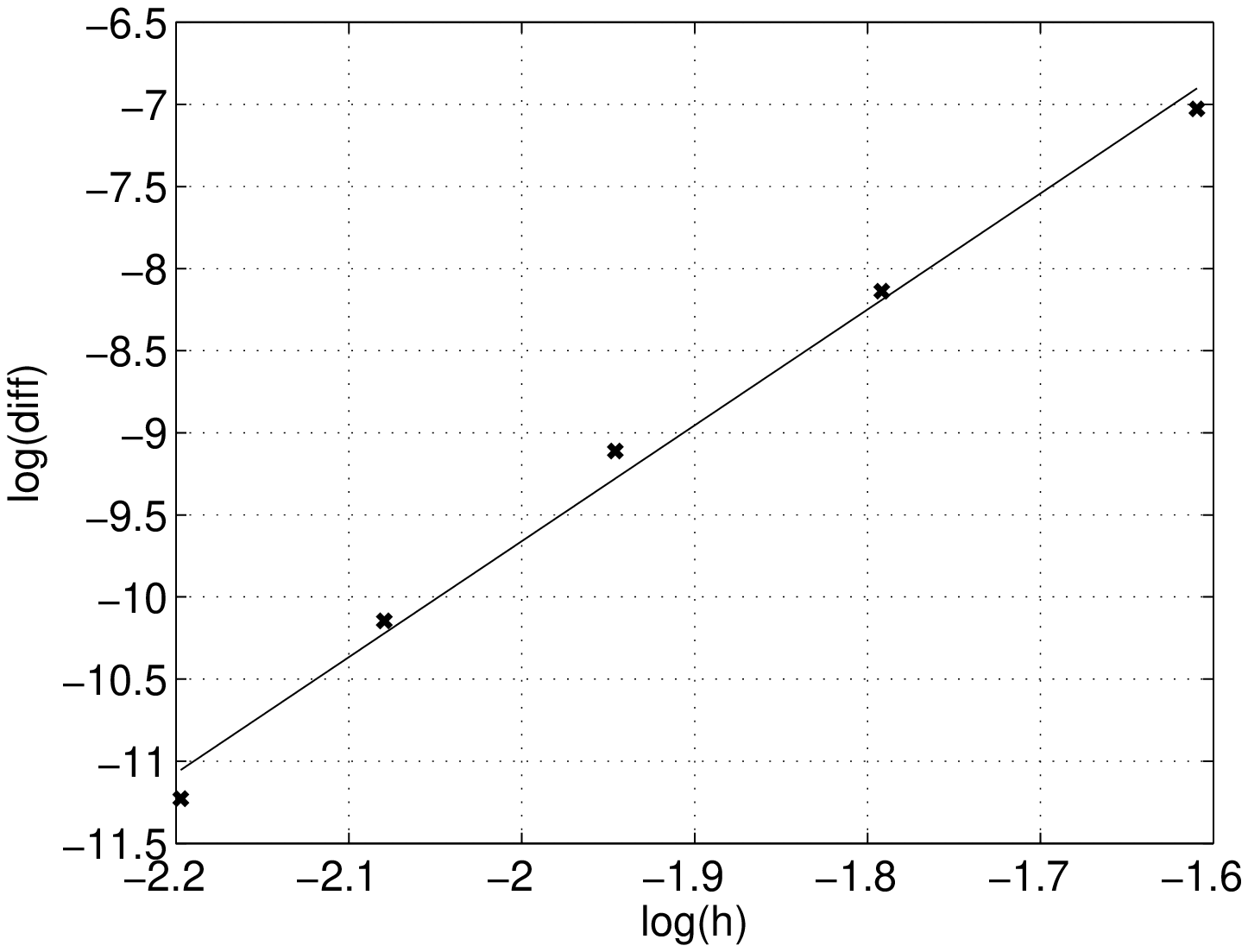} 
\vspace{-2mm}
\caption{Left: The true solution and estimates given by the three
  filtering methods. Right: The convergence of $\lim_{k \to
    \infty}\Ex{\norm{Q\tilde x_k- \hat{x}_k}_{\sX}^2}$ as $h_c \to 0$
  is shown with the x-markers. The solid line is a fitted regression
  line. The plot is in logarithmic scale.}
\label{fig:example}\end{figure}

\section{Conclusions and remarks}

When the system at hand is infinite dimensional (or its dimension is
very large), one needs to make some finite (or lower) dimensional
approximation of the system in order to be able to actually compute
something. For what comes to the Gaussian state estimation problem,
the spatial discretization introduces a bias in the Kalman filter but
the result can be improved by taking that error into account when
determining the Kalman gain.


In this paper, we derived the optimal one-step state estimator $\tilde
x_k$ for an infinite dimensional system that takes values in a
pre-defined finite dimensional subspace $\Pi_s\sX$ of the system's
state space $\sX$. The presented method also gives an operator $Q_k$
that gives $\Ex{x_k|\tilde x_k}=Q_k\tilde x_k$. This operator can be
used as a sort of post-processor of the obtained state estimate.

Sections \ref{sec:asymptotics} and \ref{sec:error} were devoted to
finding a bound for the error caused by the discretization. The error
measure is the $L^2(\Omega,\sX)$-distance between the reduced-order
state estimate $Q_k\tilde x_k$ and the full state Kalman filter
estimate $\hat x_k$, that is, $\Ex{\norm{Q_k\tilde x_k-
    \hat{x}_k}_{\sX}^2}$. It was found that this distance converges to
zero as the approximation abilities of the projection $\Pi_s$ improve.

A numerical example on temporally discretized 1D wave equation was
presented in Section~\ref{sec:example}. It was noted that the
presented method worked well even with fairly low level of
discretization. The spatial discretization was done using piecewise
linear hat functions whose approximating properties were noted to
converge with rate $\mathcal{O}(h)$ when the discretization is
refined. By Theorem~\ref{thm:error} this would imply convergence rate
$\Ex{\norm{Q_k\tilde x_k- \hat{x}_k}_{\sX}^2}=\mathcal{O}(h^2)$ for the
reduced-order state estimate. However, numerical simulations showed
that this convergence was actually of order $\mathcal{O}(h^7)$ in the
example case.

\subsection{On practical implementation} \label{sec:practice}

Even though all the computations needed for the update of the state
estimate are carried out in the finite dimensional subspace $\Pi_s\sX$
in the presented method, the offline computations needed for
determining the Kalman gains $K_k$ and the operators $Q_k$ are still
formally carried out in the infinite dimensional $\sX$. In practice,
there are very few cases where this can be done analytically, and even
then it is hardly worth the effort. A practical approach is proposed
in the example, namely introducing two computational meshes for the
problem at hand --- a fine mesh and a coarse mesh. The fine mesh
discretization is then regarded as the true system and $K_k$ and $Q_k$
are computed using this discretization. This mesh should be as fine
as reasonably possible. The online state estimation is then carried
out in the coarse mesh. Of course, the criterion for this mesh is that
the time evolution of the state estimator has to be solvable with the
available computing power in time before the next measurement arrives.

In practical implementation of the presented method, one weak point is
the computation of $Q_k$ which in theory requires computation of the
(pseudo)inverse of the $n \times n$ matrix $\tilde S_k$, see
\eqref{eq:Q_k}. As noted in Remark~\ref{rem:Q_add}, when $\tilde S_k$
is not invertible then $Q_k=\Pi^*+(I-\Pi_s)V_k\tilde S_k^+$. This
equation for $Q_k$ could also be used if the pseudoinverse is not
computed accurately, but by using some approximative or regularizing
scheme. Then the part that $Q_k$ maps to $\Pi_s \sX$ is readily taken
care of and from $V_k\tilde S_k^+$ one can compute an approximation to
a couple of the most important dimensions in the null space of $\Pi$.

We also remark that there is no guarantee that $Q_k$ and $K_k$ would
converge. Further, even if they do converge, there are no algebraic
equations for obtaining the limits directly. Thus, the only way to
obtain them is to iterate the recursive equation sufficiently many
times. However, consider the case that we are given $\Pi x_k$ and we
want to recover $x_k$. Then (assuming $\Ex{x_k}=0$) the optimal
solution is given by $\Ex{x_k|\Pi x_k}=:\widehat Q_k \Pi x_k$ where
\[
\widehat Q_k=\Pi^*+(I-\Pi_s)S_k\Pi^*(\Pi S_k \Pi^*)^+
\]
where $S_k=\Cov{x_k,x_k}$ is given by \eqref{eq:S_k}. Then we have
$x_k=\widehat Q_k \Pi x_k+v_k$
where $v_k \sim N(0,\widehat V_k)$ where 
\[
\widehat
V_k=(I-\Pi_s)S_k(I-\Pi_s)^*-(I-\Pi_s)S_k\Pi^*(\Pi S_k \Pi^*)^+\Pi S_k
(I-\Pi_s)^*.
\]
Now $S_k$ converges and the limit $S_{\infty}$ can be obtained as the
solution of the Lyapunov equation $S_{\infty}=AS_{\infty}A^*+BUB^*$.
Of course, the error $v_k$ is correlated but making the (false)
assumption that it is not, leads to an approximate reduced order error
covariance (in converged form)
\[
\begin{cases}
  \tilde P=\Pi A \widehat Q_{\infty}P\widehat Q_{\infty}^*A^*\Pi^*+\Pi
  B U B^* \Pi^*+\Pi A \widehat V_{\infty} A^* \Pi^*, \\
P=\tilde P-\tilde P \widehat Q_{\infty}^*C^*(C \widehat Q_{\infty} \tilde P \widehat Q_{\infty}^* C^*+R)^{-1}C \widehat Q_{\infty} \tilde P.
\end{cases}
\]
It was found that using this approximative state estimate worked
reasonably well in the presented example.  With the parameters on the
left in Table~\ref{tab:parameters}, the error $\norm{\widehat
  Q_{\infty} \tilde x_k-x_k}_{\sX}^2$ was in average over 500
simulations .6148 for the position variable and .8179 for the velocity
variable (cf. the right panel of Table~\ref{tab:parameters}).







\subsection{Further work}

Let us end the paper by briefly discussing topics that would require
further work. An immediate question is whether a similar result can be
obtained for the Kalman--Bucy filter, that is, for continuous time
systems. Here the discrete time systems were studied for technical
convenience but, in principle, there should not be any reasons why it
couldn't be done. For example the results of \cite{Bernstein},
\cite{Bernstein_DPS} and \cite{Galerkin_filter} were obtained in the
continuous time setting. In particular \cite{Galerkin_filter} might
give useful tools for treating this problem.

The dual problem to the Gaussian state estimation problem is the
optimal control problem for linear systems with quadratic cost
functions. A natural question is whether the results of this paper can
be translated to that problem. For example Mohammadi \emph{et al}. use
truncated eigenbasis approach to approximately solve the algebraic
Riccati equation arising from optimal control of a
diffusion-convection-reaction in \cite{Mohammadi}.

One topic that was not given much attention in this paper is the
optimality of the assumptions on the system. It is well known that the
classical Kalman filter might work just fine even though the
underlying system is not stable. We, on the other hand, used many
times the input stability of the system, \emph{i.e.}, the state
covariance is uniformly bounded by some trace class operator $S_k \le
S$. Also, we had to state as an assumption that the full state Kalman
filter is exponentially stable, that is, $\sigma(A-\hat KCA) \subset
B(0,\rho)$ for some $\rho<1$. Relaxing this assumption would be
desirable since for example strong (that is, asymptotical) stability
of the full state filter is proved in \cite[Theorem 4.2]{Horowitz_phd}
--- although under a controllability assumption that would exclude
finite dimensional control.







\subsection*{Acknowledgements}

The author has been supported by the Finnish Graduate School in
Engineering Mechanics.  The author thanks Dr. Jarmo Malinen for
valuable comments on the manuscript.

\appendix
\section{Auxiliary results}

\begin{lemma} \label{lem:L_properties} 

Define the operator ${\bf L}
  \in \mathcal{L}(\mathcal{L}^*(\sX))$ by
\begin{equation} \nonumber
{\bf L}W:=W-(A-K^{(F)}CA)W(A-K^{(F)}CA)^*
\end{equation}
where $\sigma(A-K^{(F)}CA) \subset B(0,\rho)$ for $\rho < 1$. This
operator has the following properties:
\begin{itemize}
\item[(i)] ${\bf L}$ is boundedly invertible.
\item[(ii)] If ${\bf L}W=X$, then $X \ge 0$ implies $W \ge 0$.
\item[(iii)] There exists a constant $L>0$ s.t. $\tr({\bf L}^{-1}X)
  \le L \tr(X)$ for all positive definite trace class operators $X \in
  \mathcal{L}^*(\sX)$. Denote by $L$ the smallest possible
  constant. Denote $L_0:=\sum_{j=0}^{\infty}\norm{(A-
    K^{(F)}CA)^{2j}}_{\mathcal{L}(\sX)}$. We have $L \le L_0 < \infty$.
\end{itemize}
Define also $\tilde {\bf L}W:=W-(A-K_{\infty}CA)W(A-K_{\infty}CA)^*$
where $K_{\infty}$ is the converged gain of the reduced order filter
(if it converges) and denote by $\tilde L$ the corresponding trace
bound for $\tilde {\bf L}^{-1}$.
\end{lemma}
\begin{proof}
{\it (i)}: The inverse of ${\bf L}$ is given by
\begin{equation} \label{eq:L_inv}
{\bf L}^{-1}X=\sum_{j=0}^{\infty}(A-K^{(F)}CA)^jX((A-K^{(F)}CA)^*)^j.
\end{equation}
By Gelfand's formula (see \cite[Theorem~7.5-5]{Kreyszig}), the sum
converges in operator topology because $\sigma(A-K^{(F)}CA) \subset
B(0,\rho)$ for some $\rho<1$.

{\it (ii)}: Assume that $X \in \mathcal{L}^*(\sX)$ is positive
semidefinite. From \eqref{eq:L_inv}  
it is easy to see that ${\bf L}^{-1}X$ is positive
semidefinite. Clearly also if $X$ is negative semidefinite then $W$ is
negative semidefinite.

{\it (iii)}: If $X \in \mathcal{L}^*(\sX)$ is a positive definite trace
class operator and $T \in \mathcal{L}(\sX)$ then $\tr(TXT^*) \le
\norm{T}_{\mathcal{L}(\sX)}^2 \tr(X)$. This together with
\eqref{eq:L_inv} imply {\it (iii)}.
\end{proof}

\noindent If ${\bf L}W=X_+-X_-$ where $X_+,X_- \ge 0$ then $W=W_+-W_-$
where ${\bf L}W_{\pm}=X_{\pm}$ and $W_+,W_- \ge 0$.  Of course $\tr(W)
\le \tr(W_+) \le L\tr(X_+)$. Thus, if the right hand side can be
represented as a sum of a positive definite and a negative definite
part, then only the positive definite part needs to be taken into
account when computing an upper bound for the trace of the solution.



\begin{lemma} \label{lem:DP}
  The perturbation $\Delta P$ in the proof of Theorem~\ref{thm:error}
  satisfies
\begin{equation} \nonumber
\Delta P=(A-K^{(F)}CA)\Delta P(A-K^{(F)}CA)^*+E_1+E_2+h_1(\Delta P)+h_2(\Delta P)
\end{equation}
\begin{equation} \label{eq:DP}
={\bf L}^{-1}\left(E_1+E_2+h_1(\Delta P)+h_2(\Delta P) \right)
\end{equation}
where 
\begin{equation} \nonumber
E_1=(I-K^{(F)}C)AMA^*(I-K^{(F)}C)^*,
\end{equation}
\begin{equation} \nonumber E_2=-(I-K^{(F)}C)AMA^*C^*\left( C(\tilde
    P^{(F)}+AMA^*)C^*+R \right)^{-1}CAMA^*(I-K^{(F)}C)^*
\end{equation}
\begin{equation} \nonumber
+K^{(F)}CAMA^*C^*\left( C(\tilde P^{(F)}+AMA^*)C^*+R \right)^{-1}CAMA^*C^*K^{(F)*},
\end{equation}
\begin{equation} \nonumber h_1(\Delta P)=\Delta KCA\Delta
  P(A-K^{(F)}CA)^*+(A-K^{(F)}CA)\Delta P (\Delta KCA)^*+\Delta KCA\Delta
  P(\Delta KCA)^* 
\end{equation}
where $\Delta K=K^{(F)}-K^{(b)}$, and
\begin{align*}  h_2(\Delta P)=&-(A-K^{(b)}CA)\Delta P A^*C^* \!  \left(
    C(\tilde P^{(F)} \! +AMA^* \! +A\Delta PA^*)C^* \! +R\right)^{-1} \! \times \\
& \qquad \qquad \times CA\Delta P (A-K^{(b)}CA)^*.
\end{align*}

Alternatively, the equation \eqref{eq:DP} can be written as
\begin{equation} \label{eq:DP_alt} \Delta P=(A-K^{(b)}CA)\Delta
  P(A-K^{(b)}CA)^*+E_1+E_2+h_2(\Delta P).
\end{equation}

The perturbation of the Kalman gain is given by
\begin{align*} \Delta K=&\left((\tilde
    P^{(F)}+AMA^*)C^*\left(C(\tilde P^{(F)}+AMA^*)C^*+R
    \right)^{-1}C-I\right)\times \\
&\qquad \qquad \times AMA^*C^*(C\tilde P^{(F)}C^*+R)^{-1}.
\end{align*}

\end{lemma}
\noindent For a proof, see \cite[Lemma~2.1]{Sun}. There everything is
finite-dimensional but the proof of this Lemma is based on just
algebraic manipulation and it holds also in the infinite-dimensional
setting. Note that the matrix $C(\tilde P^{(F)}+M+A\Delta PA^*)C^*+R$ is
invertible because $C(\tilde P^{(F)}+M+A\Delta PA^*)C^* \ge 0$ and $R >
0$. In the proof of \cite[Lemma~2.1]{Sun}, some additional assumptions
on the perturbations is needed to guarantee the invertibility of the
corresponding matrix (denoted by $\tilde C$ there). To get 
\eqref{eq:DP_alt}, note that
\begin{equation} \nonumber h_1(\Delta P)=(A-K^{(b)}CA)\Delta
  P(A-K^{(b)}CA)^*-(A-K^{(F)}CA)\Delta P(A-K^{(F)}CA)^*. 
\end{equation}
For the last part, see in
particular \cite[Eq.~(A.8)]{Sun}.

\vspace{3mm}


\end{document}